\documentclass[11pt,reqno]{amsart}
\usepackage{amssymb,latexsym}
\usepackage{graphicx}
\usepackage{mathtools}
\usepackage{color}
\usepackage[hyphens]{url}
\usepackage[T1]{fontenc}
\usepackage{hyperref}
\usepackage{amsmath}
\usepackage{mathdots}
\usepackage{breqn}
\usepackage[toc,page]{appendix}
\usepackage{url}    
\usepackage{breqn}
\usepackage{hyperref}
\usepackage{breakurl}
\newcommand{\bburl}[1]{\textcolor{blue}{\url{#1}}}

\calclayout

\makeatletter
\newcommand{\monthyear}[1]{%
  \def\@monthyear{\uppercase{#1}}}
\newcommand{\volnumber}[1]{%
  \def\@volnumber{\uppercase{#1}}}
\AtBeginDocument{%
\def\ps@plain{\ps@empty
  \def\@oddfoot{\@monthyear \hfil \thepage}%
  \def\@evenfoot{\thepage \hfil \@volnumber}}
\def\ps@firstpage{\ps@plain}
\def\ps@headings{\ps@empty
  \def\@evenhead{%
    \setTrue{runhead}%
    \def\thanks{\protect\thanks@warning}%
    \uppercase{\ }\hfil}%
  \def\@oddhead{%
    \setTrue{runhead}%
    \def\thanks{\protect\thanks@warning}%
    \hfill\uppercase{Various Sequences from Counting Subsets}}%
  \let\@mkboth\markboth
  \def\@evenfoot{%
    \thepage \hfil \@volnumber}%
  \def\@oddfoot{%
    \@monthyear \hfil \thepage}%
  }%
\footskip=25pt
\pagestyle{headings}%
}
\makeatother

\theoremstyle{plain}
\numberwithin{equation}{section}
\newtheorem{thm}{Theorem}[section]

\newcommand{\seqnum}[1]{\href{https://oeis.org/#1}{\underline{#1}}}

\newcommand{\ignore}[1]{}

\newcommand{\lf}{\lfloor}
\newcommand{\rf}{\rfloor}

\newcommand{\floor}[1]{\lf {#1} \rf}

\newcommand\be{\begin{eqnarray}}
\newcommand\ee{\end{eqnarray}}
\newcommand\bea{\begin{eqnarray}}
\newcommand\eea{\end{eqnarray}}
\newcommand\ben{\begin{enumerate}}
\newcommand\een{\end{enumerate}}

\newtheorem{cor}[thm]{Corollary}

\newtheorem{prop}[thm]{Proposition}

\newtheorem{defi}[thm]{Definition}

\newtheorem{rek}[thm]{Remark}

\begin{document}

\monthyear{}
\volnumber{Volume, Number}
\setcounter{page}{1}
\title{Various Sequences from Counting Subsets}

\author{H\`ung Vi\d{\^e}t Chu}

\address{Department of Mathematics, University of Illinois at Urbana-Champaign, Urbana, IL 61820} \email{hungchu2@illinois.edu}

\date{\today}

\begin{abstract}
As $n$ varies, we count the number of subsets of $\{1,2,\cdots,n\}$ under different conditions and study the sequence formed by these numbers. 
\end{abstract}

\thanks{The author is thankful for the anonymous referee's comments that improved the exposition of this paper.}

\maketitle
\section{Introduction}

We define the $\alpha$-Schreier condition. 
Given natural number $\alpha$, a set $S$ is said to be \textit{$\alpha$-Schreier} if $\min S/\alpha \ge |S|$, where $|S|$ is the cardinality of $S$. Schreier used 1-Schreier sets to solve a problem in Banach space theory \cite{S}. These sets were also independently discovered in combinatorics and are connected to Ramsey-type theorems for subsets of $\mathbb{N}$. Next, we define the $\beta$-Zeckendorf condition. In 1972, Zeckendorf proved that every positive integer can be uniquely written as a sum of non-consecutive Fibonacci numbers \cite{Z}. We focus on the important requirement for uniqueness of the Zeckendorf decomposition; that is, our set contains no two consecutive Fibonacci numbers. We generalize this condition to a finite set of natural numbers. 

\begin{defi}
Let $S = \{s_1,s_2,\ldots, s_k\}$ $(s_1<s_2<\cdots<s_k)$ for some $k\in \mathbb{N}_{\ge 2}$. The \textit{difference set} of $S$, denoted by $D(S)$, is $\{s_2-s_1,s_3-s_2,\ldots,s_k-s_{k-1}\}$. The difference set of the empty set and a set with exactly one element is empty. 
\end{defi}

\begin{defi}
Fix a natural number $\beta$. A finite set $S$ of natural numbers is \textit{$\beta$-Zeckendorf} if $\min D(S)\ge \beta$; that is, each pair of numbers in $S$ is at least $\beta$ apart. The empty set and a set with exactly one element vacuously satisfy this condition. 
\end{defi}

Chu et al. proved the linear recurrence of the sequence obtained by counting subsets of $\{1,2,\ldots, n\}$ that are $\alpha$-Schreier \cite{Ch2}. In particular, \cite[Theorem 1.1]{Ch2} states that the recurrence has order $\alpha+1$. On the other hand, it is well known that the sequence obtained by counting subsets of $\{1,2,\ldots, n\}$ that are $\beta$-Zeckendorf has a linear recurrence of order $\beta$. A notable example is $\beta = 2$, which gives the Fibonacci sequence. A natural extension of these results is to consider sets that are both $\alpha$-Schreier and $\beta$-Zeckendorf. For each $n\in\mathbb{N}$, define $$a_{\alpha,\beta, n}:\ =\ \#\{S\subset \{1, 2, \ldots, n\}\,:\, S \text{ is }\alpha\text{-Schreier} \text{ and }\beta\text{-Zeckendorf}\}.$$ 
Our first result shows a linear recurrence for this sequence $(a_{\alpha, \beta, n})$.
\begin{thm}\label{m1}
Fix natural numbers $\alpha$ and $\beta$. For $n\ge 1$, we have
\begin{align*}
a_{\alpha, \beta, n} \ =\ \begin{cases} 1, & \text{ for } n \le \alpha - 1;\\
n-\alpha+2, & \text{ for } \alpha \le n \le 2\alpha+\beta-1;\\
a_{\alpha, \beta, n-1} + a_{\alpha, \beta, n - (\alpha+\beta)}, & \text{ for }  n\ge 2\alpha + \beta.
\end{cases}\end{align*}
\end{thm}
\begin{rek}
Theorem \ref{m1} says that the order of our recurrence relation is the sum $\alpha+\beta$. Substituting $\beta = 1$, we have \cite[Theorem 1.1]{Ch2}. Interestingly, the number of 1's in the sequence is independent of $\beta$. 
\end{rek}

The next results involve the Fibonacci sequence. Let the Fibonacci sequence be $F_0 = 0$, $F_1 = 1$, and $F_n = F_{n-1}+F_{n-2}$ for all $n\ge 2$. Let $(H_n)_{n\ge 0}$ be the sequence obtained by applying the partial sum operator twice to the Fibonacci sequence. In particular, 
$$H_n \ =\ \sum_{i=0}^n (n+1-i)F_i.$$
The first few terms of $(H_n)$ are $0, 1, 3, 7, 14, 26$, and $46$. We prove the following identity.

\begin{prop}\label{propi}
For $n\ge 0$, we have \begin{align}\label{Fibi}F_{n+4}\ =\ H_{n} + n + 3.\end{align}
\end{prop}
We then use the identity to prove the following theorem. 
\begin{thm}\label{m2}
Let $(a_n)_{n\ge 1}$ be the number of subsets of $\{1, 2,\ldots, n\}$ that 
\begin{itemize}
    \item [(i)] have at least 2 elements; and 
    \item [(ii)] have their difference sets only contain odd numbers,
\end{itemize}
then $a_n = H_{n-1}$. 
\end{thm}

Our last result is a companion of \cite[Theorem 8]{Ch1}, which considers subsets of $\{1,2,\ldots, n\}$ whose difference set only contains odd numbers. Surprisingly, the number of such subsets is related to the Fibonacci sequence. For convenience, we include the theorem below. 
\begin{thm}\label{m'}
Fix $n\in \mathbb{N}$. The number of subsets of $\{1,2,\ldots,n\}$ 
\begin{enumerate}
\item that contain $n$ and whose difference sets only contain odd numbers is $F_{n+1}$,
\item whose difference sets only contain odd numbers (the empty set and sets with exactly one element vacuously satisfy this requirement) is $F_{n+3}-1$.
\end{enumerate}
\end{thm}

To complete the picture, we consider subsets whose difference set only contains even numbers. 

\begin{thm}\label{m3}
Fix $n\in\mathbb{N}$. The number of subsets of $\{1, 2, \ldots, n\}$ 
\begin{itemize}
\item [1.] that contain $n$ and whose difference sets only contain even numbers is $2^{\floor{(n-1)/2}}$,

\item [2.] whose difference sets only contain even numbers (the empty set and sets with exactly one element vacuously satisfy this requirement) is 
$$\begin{cases}3\cdot 2^{(n-1)/2}-1, & \text{ if } n \text{ is odd;}\\
2\cdot 2^{n/2} - 1, & \text{ if } n \text{ is even}.\end{cases}$$
\end{itemize}
\end{thm}
The following corollary is immediate. 
\begin{cor}
Let 
\begin{align*}\mathcal{S}_n & := \{S\subset \{1,2,\ldots, n\}: D(S) \text{ only has odd numbers or only has even numbers}\};\\
\mathcal{S}_{n,1} & := \{S\subset \{1,2,\ldots, n\}: D(S) \text{ only has odd numbers}\};\\
\mathcal{S}_{n,2} & := \{S\subset \{1,2,\ldots, n\}: D(S) \text{ only has even numbers}\}.
\end{align*}
Then $\lim_{n\rightarrow \infty} \frac{|\mathcal{S}_{n,1}|}{|\mathcal{S}_n|} = 1$; that is, as $n\rightarrow\infty$, almost all sets in $S_n$ have their difference sets only contain odd numbers. 
\end{cor}
\begin{proof}
Since $3\cdot 2^{(n-1)/2} > 2\cdot 2^{n/2}$ and by Theorem \ref{m'}, it suffices to prove that 
\begin{align*}
    \lim_{n\rightarrow \infty} \frac{3\cdot 2^{(n-1)/2}-1}{F_{n+3}-1}\ =\ 1,
\end{align*}
which we can prove easily by using the Binet's formula for $F_{n+3}$. 
\end{proof}

\begin{rek}
Intuitively, the above corollary says that sets in $\mathcal{S}_{n,1}$ dominate sets in $\mathcal{S}_{n,2}$. The reason is that for a set in $\mathcal{S}_1$, the difference between consecutive elements can be as small as $1$, which gives us much more freedom in constructing such a set than a set in $\mathcal{S}_2$. 
\end{rek}

Section 2 is devoted to proofs of our main results, while Section 3 generalizes Proposition \ref{propi} and raises several questions for future research. 
\section{Proofs}
\begin{proof}[Proof of Theorem \ref{m1}]
For $n\le \alpha-1$, the only subset of $\{1,2,\ldots, n\}$ that is $\alpha$-Schreier is the empty set, which is also $\beta$-Zeckendorf. Hence, $a_{\alpha,\beta, n} = 1$. 

Consider $\alpha \le n \le 2\alpha+\beta-1$. Let $S\subset \{1,2,\ldots, n\}$ be $\alpha$-Schreier and $\beta$-Zeckendorf. Suppose that $|S| \ge 2$. Because $\min S/\alpha \ge |S|\ge 2$, we have $\min S\ge 2\alpha$. Because $S$ is $\beta$-Zeckendorf, the other elements in $S$ must be at least $2\alpha+\beta$, which contradicts $n\le 2\alpha+\beta-1$. Hence, either $S = \emptyset$ or $S = \{k\}$ for $\alpha\le k\le n$. Therefore, $a_{\alpha, \beta, n} = n-\alpha+2$. 

Finally, consider $n\ge 2\alpha + \beta$. Let 
\begin{align*}
    \mathcal{A}:\  &=\ \{S\subset \{1,2, \ldots, n\}\,:\, S \text{ is }\alpha\text{-Schreier} \text{ and } \beta\text{-Zeckendorf}\text{ and } \max S < n\};\\
    \mathcal{B}:\  &=\  \{S\subset \{1,2, \ldots, n\}\,:\, S \text{ is }\alpha\text{-Schreier} \text{ and } \beta\text{-Zeckendorf} \text{ and }\max S = n\}.
\end{align*}
 Clearly, $\mathcal{A} = \{S\subset \{1,2,\ldots, n-1\}\,:\,  S \text{ is }\alpha\text{-Schreier} \text{ and } \beta\text{-Zeckendorf}\}$. Hence, $|\mathcal{A}| = a_{\alpha, \beta, n-1}$. It suffices to prove $|\mathcal{B}| = a_{\alpha, \beta, n-(\alpha+\beta)}$. We show a bijection between $\mathcal{B}$ and $\mathcal{C}: = \{S\subset \{1,2,\ldots, n-(\alpha+\beta)\}\,:\,  S \text{ is }\alpha\text{-Schreier} \text{ and } \beta\text{-Zeckendorf}\}$.
 
 Given a set $S$ and $k\in \mathbb{N}$, we let $S-k: = \{s-k\,:\, s\in S\}$. Define the function $f: \mathcal{B}\rightarrow\mathcal{C}$ such that 
 \begin{align*}
 f(S) \ =\ \begin{cases} \emptyset, & \text{ if } S = \{n\};\\
 S\backslash\{n\} - \alpha, & \text{ if } |S| > 1.
 \end{cases}
 \end{align*}
We show that $f$ is well-defined. If $|S| > 1$, we have $$\min f(S) \ =\ \min (S\backslash\{n\} - \alpha) \ =\ \min S - \alpha \ \ge\ \alpha|S| - \alpha \ =\ \alpha|f(S)|.$$
Hence, $f(S)$ is $\alpha$-Schreier. Because $S$ is $\beta$-Zeckendorf, $f(S)$ is also $\beta$-Zeckendorf. Lastly, we have
$$\max f(S) \ =\ \max (S\backslash\{n\})-\alpha \ \le \ (n-\beta)-\alpha \ =\ n-(\beta+\alpha).$$
Therefore, $f(S)\in \mathcal{C}$. We know that $f$ is injective by definition and thus, $|\mathcal{B}|\le |\mathcal{C}|$. Next, define the function $g: \mathcal{C}\rightarrow \mathcal{B}$ such that $g(S) = (S+\alpha)\cup \{n\}$. Because $S$ is $\beta$-Zeckendorf and $\max S\le n-(\alpha+\beta)$, we know that $g(S)$ is also $\beta$-Zeckendorf. To see why $g(S)$ is $\alpha$-Schreier, we observe that
$$\min g(S) \ =\  \min S + \alpha \ \ge\ \alpha(|S|+1)\ = \ \alpha|g(S)|.$$
Hence, $g$ is well-defined and is injective by definition. Therefore, $|\mathcal{B}|\ge |\mathcal{C}|$. We conclude that $|\mathcal{B}| = |\mathcal{C}|$, which completes our proof. 
\end{proof}

\begin{proof}[Proof of Proposition \ref{propi}] We prove by induction. Clearly, the identity holds for $n = 0$. Suppose the identity holds for $n = k$ for some $k\ge 0$; that is, 
$F_{k+4} \ =\ H_k + k + 3$.
We show that $F_{k+5} \ =\ H_{k+1} + k+4$. We have
\begin{align*}
    F_{k+5}\ =\ F_{k+4} + F_{k+3} &\ =\ H_k + k + 3 + F_{k+3}\\
    &\ =\ \left(H_{k+1} - \sum_{i=0}^{k+1} F_i\right) + k + 3 + F_{k+3}\\
    &\ =\ \left(H_{k+1}+k+4\right) + \left(F_{k+3} - \sum_{i=0}^{k+1} F_i - 1\right).
\end{align*}
It is well known that $F_{k+3} - \sum_{i=0}^{k+1} F_i - 1 = 0$; therefore, $F_{k+5} = H_{k+1} + k + 4$. This completes our proof.  
\end{proof}

\begin{proof}[Proof of Theorem \ref{m2}] By Theorem \ref{m'} and \eqref{Fibi}, we have 
\begin{align*}a_n\  =\ F_{n+3} - 2 - n \ =\ (H_{n-1}+n+2)-2-n \ =\ H_{n-1}.
\end{align*}
This completes our proof. 
\end{proof}

\begin{proof}[Proof of Theorem \ref{m3}] We prove the first item by induction. Let $P_n$ (and $O_n$, resp) be the number of subsets of $\{1,2,\ldots, n\}$ (and the set of subsets of $\{1, 2,\ldots,n\}$, resp) that satisfy our requirement. 

\textit{Base cases:} For $n = 1$, $\{1\}$ is the only subset of $\{1\}$ that satisfies our requirement. Hence, $P_1 = 1 = 2^{\floor{(1-1)/2}}$. Similarly, $O_2 = \{\{2\}\}$ and $P_2 = 1 = 2^{\floor{(2-1)/2}}$.

\textit{Inductive hypothesis:} Suppose that there exists $k\ge 2$ such that for all $n\le k$, we have $P_n = 2^{\floor{(n-1)/2}}$. We show that $P_{k+1} = 2^{\floor{k/2}}$. Observe that unioning a set in $O_{k+1-2i}$ with $k+1$ produces a set in $O_{k+1}$ and any set in $O_{k+1}$ is of the form of a set in $O_{k+1-2i}$ plus the element $k+1$. Therefore, 
\begin{align*}
    P_{k+1} \ =\ |O_{k+1}| \ =\ 1 + \sum_{1\le i< (k+1)/2} |O_{k+1-2i}| \ =\  1 + \sum_{1\le i< (k+1)/2} P_{k+1-2i}.
\end{align*}
The number $1$ accounts for the set $\{k+1\}$.
If $k$ is odd, we have
\begin{align*}
    P_{k+1} &\ =\ 1 + P_{k-1} + P_{k-3} + \cdots + P_2 \\
    &\ =\ 1 + 2^{\floor{(k-2)/2}}+2^{\floor{(k-4)/2}} + \cdots + 2^{\floor{1/2}}\\
    &\ =\ 1 + 2^{(k-3)/2}+2^{(k-5)/2} + \cdots + 2^{0/2} \ = \ 2^{(k-1)/2}\ =\ 2^{\floor{k/2}}.\\
\end{align*}
Similarly, if $k$ is even, we also have $P_{k+1} = 2^{\floor{k/2}}$. This completes our proof of the first item. The second item easily follows from the first by noticing that the number of subsets of $\{1,2,\ldots, n\}$ whose difference sets only contain even numbers is equal to $1+ \sum_{k=1}^n |O_k| = 1+\sum_{k=1}^n 2^{\floor{(k-1)/2}}$, where the number $1$ accounts for the empty set. It is an easy exercise to show that this formula and the formula given in item 2 are the same. 
\end{proof}
\section{Generalizations and Questions}
In this section, we generalize Proposition \ref{propi} and raise two questions for future research. For each $n\ge 2$, define the sequence $(F_{n,m})_{m\ge 0}$ as follows: $F_{n,0} = 0, F_{n,1} = \cdots = F_{n,n} = 1$, and $F_{n,m} = F_{n,m-1} + F_{n,m-n}$ for $m\ge n+1$. Let $(K_{n,m})$ and $(H_{n,m})$ be the sequence obtained by applying the partial sum operation to $(F_{n,m})$ once and twice, respectively. For example, when $n = 3$, we have Table 1.

\begin{tabular}{c|ccccccccccccc}
$m$ & $0$ & $1$ & $2$ & $3$ & $4$ & $5$ & $6$ & $7$ & $8$ & $9$ & $10$ & $11$ & $12$\\
\hline
$F_{n,m}$ & 0 & 1 & 1 & 1 & 2 & 3        & $4$ &    $6$ &   $9$ & $13$ & $19$ & $28$ & $41$ \\
$K_{n,m}$ & 0 & 1 & 2 & 3 & $5$ & $8$    & $12$ &   $18$&   $27$ & $40$ & $59$ & $87$ & $128$\\
$H_{n,m}$ & 0 & 1 & 3 & 6 & $11$ & $19$   & $31$ &   $49$&   $76$ & $116$ & $175$ & $262$ & $390$\\
\end{tabular}
\begin{center}
Table 1. The sequences $(F_{3,m}), (K_{3,m})$, and $(H_{3,m})$ for $0\le m\le 12$.
\end{center}

The first row is the sequence \seqnum{A000930} in the {\it On-Line Encyclopedia of Integer Sequences} (OEIS) \cite{OEIS}; the second 
row is \seqnum{A077868} and the third row is \seqnum{A050228}. Many thanks to N. J. A. Sloane for pointing out a miscalculation in Table 1 of the earlier version.

The following proposition generalizes Proposition \ref{propi}. 
\begin{prop}
For $n\ge 2$ and $m\ge 0$, we have
\begin{itemize}
    \item[(1)] $\sum_{i=0}^{k+1} F_{n,i} = F_{n,k+1+n}-1$ \mbox{ for }$k\ge 0$,
    \item[(2)] $F_{n, m+2n} = H_{n,m} + m + (n+1)$.
\end{itemize}
\end{prop}
\begin{proof}
We prove item (1). Fix $n\ge 2$ and $k\ge 0$. We have
\begin{align*}
    F_{n,k +1+n} - \sum_{i=0}^{k+1} F_{n,i} - 1 &\ =\ (F_{n,k+1+n} - F_{n,k+1}) - \sum_{i=0}^k F_{n,i} - 1\\
    &\ =\ F_{n, k + n} - \sum_{i=0}^k F_{n,i} - 1\\
    &\ =\ (F_{n, k+n} - F_{n,k}) - \sum_{i=0}^{k-1} F_{n,i} - 1\\
    &\ =\ F_{n, k+n-1} - \sum_{i=0}^{k-1} F_{n,i} - 1\\
    &\ =\ \ldots \ =\ F_{n, n-1} - 1 \ =\ 0.
\end{align*}
Hence, we have $F_{n,k+1+n} - \sum_{i=0}^{k+1}F_{n,i} - 1 = 0$, so $\sum_{i=0}^{k+1} F_{n,i} = F_{n,k+1+n} - 1 $. 

Next, we prove item (2). Fix $n\ge 2$. We prove by induction. \textit{Base case:} for $m = 0$, the identity is equivalent to $F_{n,2n} = n+1$, which is true. \textit{Inductive hypothesis:} suppose that the identity is true for all $0\le m\le k$ for some $k\ge 0$. We want to show that it is true for $m = k+1$. We have
\begin{align*}
    F_{n, k+1+2n} &\ =\ F_{n, k+2n} + F_{n, k+1+n}\\
    &\ =\ (H_{n,k}+k+(n+1)) + F_{n, k+1+n}&\mbox{ by the inductive hypothesis}\\
    &\ =\ (H_{n,k}+F_{n,k+1+n}-1) + (k+1) + (n+1)\\
    &\ =\ \left(H_{n,k} + \sum_{i=0}^{k+1} F_{n,i}\right) + (k+1) + (n+1)&\mbox{ by item (1)}\\
    &\ =\ H_{n,k+1} + (k+1) + (n+1).
\end{align*}
This completes our proof. 
\end{proof}

Theorem \ref{m2} shows that $(H_{2,m})$ is related to the number of certain subsets of $\{1, 2, \ldots, n\}$; however, the author is unable to find such a combinatorial perspective of the sequence $(H_{n,m})$ when $m>2$. Is there a connection between the sequence $(H_{n,m})$ and the number of subsets of $\{1,2, \ldots, n\}$ restricted to certain conditions as in Theorem \ref{m2}?

Fix $k\ge 2$. Another way to generalize Theorem \ref{m2} is to look at the sequence formed by counting subsets of $\{1, 2, \ldots, n\}$ satisfying two conditions: (i) have at least $k$ elements, and (ii) have their difference sets only contain odd numbers. When $k = 2$, Theorem \ref{m2} connects the sequence obtained by counting subsets to the Fibonacci sequence; however, the author is unable to find such a connection for bigger values of $k$. For example, when $k = 3$, the sequence we obtain is $0,0,1,3,8,17,34,63,113,196,334,560,\ldots$. Is there a neat relation among terms in this sequence?

\begin{rek}\normalfont
The sequence $0,0,1,3,8,17,34,63,113,196,334,560,\ldots$ was recently added to OEIS by N.J. A. Sloane, and its recurrence relation was discovered by A. Heinz (see \seqnum{A344004}).
\end{rek}


\newcommand{\etalchar}[1]{$^{#1}$}

\ \\

\noindent MSC2010: 11B39

\end{document}